%% file: prisoners_and_hats_final_as_published_aplimat_v3.tex
\documentclass[reqno,12pt]{amsart}

\usepackage[utf8]{inputenc}

\input{packages}

\newtheorem{defin}{Definition}

\newtheorem{lemma}[defin]{Lemma}
\newtheorem{prop}[defin]{Proposition}
\newcommand{\str}[1]{\StrLen{#1}[\tmplength]\ifthenelse{1=\tmplength{}}{\mathcal{#1}}{\underline{#1}}}

\newcommand{\M}{\mathcal{M}}

\newcommand{\stru}[2]{\lara{#1,#2}} % struktura
\newcommand{\thn}[1]{\mathrm{#1}} %theory name

\newcommand{\DeLOs}[2]{{\tiny\ifthenelse{\equal{#1}{#2}}{\thn{DeLO}^{#1}}{\ifthenelse{\equal{#2}{}}{\thn{DeLO}^{#1}}{\thn{DeL0}^{\vpair{$#1$}{$#2$}}}}}}

\newcommand{\limp}{\rightarrow}
\newcommand{\IFF}{\Leftrightarrow}
\newcommand{\PAK}{\Rightarrow}
\newcommand{\KAP}{\Leftarrow}

\newcommand{\N}{\mathbb{N}}
\newcommand{\Z}{\mathbb{Z}}
\newcommand{\Q}{\mathbb{Q}}

\newcommand{\lr}[1]{\{#1\}}

\newcommand{\lara}[1]{\langle #1\rangle}

\newcommand{\set}[2]{\lr{#1;#2}}

\newcommand{\sdiff}{-}
\newcommand{\suni}{\cup}

\newcommand{\Suni}{\bigcup}

\newcommand{\sbs}{\subseteq}

\newcommand{\fset}[2]{{}^{#1}#2} %mnozina vsech funkci z #1 do #2
\newcommand{\factor}[2]{#1/{#2}}
\newcommand{\poten}[1]{\mathcal{P}(#1)} % mnozinova potence

\newcommand{\id}[1]{id_{#1}}

\newcommand{\restr}{\upharpoonright}

\newcommand{\uvz}[1]{``#1''} % anglicke uvozovky

\usepackage{color}

\newcommand{\menumf}{\mathrm{m}}
\newcommand{\menum}[1]{\menumf(#1)}
\newcommand{\guessf}{g}
\newcommand{\guess}[1]{\guessf(#1)}
\newcommand{\hnsa}[4]{\mathtt{hNsA}(#1,#2,#4)} % hear nothing, see all
\newcommand{\hnsf}[4]{\mathtt{hNsF}(#1,#2,#4)} % hear nothing, see forward
\newcommand{\hbsf}[4]{\mathtt{hBsF}(#1,#2,#4)} % hear backward, see forward

\begin{document}

\title{A note on the problem of prisoners and hats}

\abstract{We study the famous mathematical puzzle of prisoners and hats. We introduce a framework in which various variants of the problem can be formalized. We examine three particular versions of the problem (each one in fact a class of problems) and completely characterize them as to (non)existence of winning strategies.}


   
%\author{Jana Glivick\' a}
%\address{Jana Glivick\' a: Department of Theoretical Computer Science and Mathematical Logic, Faculty of Mathematics and Physics, Charles University, Malostranské náměstí~25, 118~00 Praha~1, Czech Republic}
          
\author{\textsc{Petr Glivick\'{y}}}
\address{\textsc{Petr Glivick\'{y}}: Department of Mathematics, Faculty of Informatics and Statistics, University of Economics, Prague, Ekonomická~957, 148~00 Praha~4, Czech Republic}
}

\thanks{This paper was processed with contribution of long term institutional support of research activities by Faculty of Informatics and Statistics, University of Economics, Prague.}

\maketitle

\section{Introduction}
Problems that can be collectively denoted as \uvz{puzzles of prisoners and hats} have been studied since 1960's (see e.g. \cite{Gar61}). The common idea of all such problems can be informally stated as follows (we give a more precise formulation below): A group of prisoners is asked to guess colors of hats they are wearing. In general, of course, they are not allowed to look at their own hats, but may be allowed to see some of the hats that others are wearing and hear the guesses of some of their colleagues. The goal of the prisoners usually is to maximize the number of correct guesses. No communication among the prisoners is allowed since the moment the hats are placed on their heads. However they have an opportunity to devise a common strategy in advance.

This may seem as just a recreational mathematical puzzle. Nevertheless certain variants of the problem have been studied and applied in contexts as distinct as coding theory \cite{Bus12}, set theory \cite{GP76} or theory of auctions \cite{AFGHIS05}.

A recent and rather informal introduction to the topic (focused on the infinite case) is \cite{HT08}. The finite case is studied in more detail in \cite{BHKL08}. An extensive treatise on the subject is \cite{HTbook}.

\medskip

In this paper we describe a framework in which the various versions of the problem of prisoners and hats can be formalized in a unified way (section \ref{sect:theproblem}). We use this framework to classify some of the important cases with regard to existence of a winning strategy for the prisoners. In particular we will be interested in the \uvz{hear nothing, see all} case (section \ref{sect:hnsa}) and two cases where prisoners can be thought of as standing in a line --- the \uvz{hear nothing, see forward} case (section \ref{sect:hnsf}) and the \uvz{hear backward, see forward} case (section \ref{sect:hbsf}). This, of course, is just the tip of the iceberg. A complete classification of prisoners and hats problems is an extensive but interesting project and an inspiration for further research.

\newpage
\subsection{Preliminaries and notation}

\subsubsection{Basic notation}
$\fset{x}{y}$ denotes the set $\set{f}{f:x\rightarrow y}$ of all functions from $x$ to $y$. $\id{X}$ denotes the identity function on $X$. A natural number $n$ and a function $f$ are understood in the usual set-theoretical sense, i.e. $n=\lr{0,\ldots,n-1}$ and $f=\set{(x,y)}{f(x)=y}$.

\subsubsection{Directed graphs}
In a directed graph $G = \stru{V}{E}$ a direction of an edge $(v',v)\in E$ is understood as from $v'$ to $v$.
A directed path in $G$ is also called ascending path, while a descending path in $G$ is a directed path in the graph $\stru{V}{E^{-1}}$.

It is easy to see that if $G$ has no infinite descending paths, then there is a well-ordering $\prec\supseteq E$ on $V$.

\subsection{The prisoners and hats problem}
\label{sect:theproblem}
The problem we want to investigate is based on the following story: 

\subsubsection{A story: Mathematicians and an evil hatter}
A nonempty set $M$ of mathematicians is captured by an evil hatter who announces that the next day they will play a hat guessing game with their freedom (if they win) or death (if they lose) at stake. 
The hatter will place differently colored hats, with colors coming from a nonempty set $C$, onto mathematicians heads and the unlucky prisoners will be forced to guess what color they will be wearing.
The malevolent captor reveals some details about the rules in advance:

\medskip
 
The game starts with the hatter choosing an assignment $a:M\rightarrow C$ from certain subset $A\sbs\fset{M}{C}$ of allowed assignments of hat colors to the mathematicians. He then places hats accordingly onto mathematicians heads and decides that a mathematician $m$ will be allowed to see the hat his colleague $m'$ is wearing if and only if $(m',m)\in S$ for certain binary relation $S\sbs M^2$.

%He then follows a well-ordering $\prec$ from certain structure $\stru{T}{\prec,H}$ whose elements are labeled by mathematicians via some $\menumf:T\rightarrow M$ (not necessarily one-to-one nor onto) and $H\sbs\prec$ is a directed graph with no infinite descending paths.

The hatter then follows some well-ordering $\prec$ that extends the canonical partial ordering of certain directed graph $\stru{T}{H}$ with no infinite descending 
%(i.e. in the direction opposite to the direction of edges) 
paths whose vertices are labeled by mathematicians via some $\menumf:T\rightarrow M$ (not necessarily one-to-one nor onto) and at time $t\in T$ he asks the mathematician $\menum{t}$ for her fateful guess (a color from $C$) which he then records as $\guess{t}$. Before $\menum{t}$ is asked at time $t$, she will be allowed to hear the record of all guesses $\guess{t'}$ such that $(t',t)\in H$. After her guess, $\menum{t}$'s memory is erased so that she doesn't remember what she heard nor what her own guess was. (It is easy to see that from the point of view of the mathematicians, the choice of $\prec\supseteq H$ is irrelevant.)

%He then follows a linear ordering $\stru{T}{\prec}$ whose elements are labeled by mathematicians via some $\menumf:T\rightarrow M$ (not necessarily one-to-one nor onto) and at time $t\in T$ he asks the mathematician $\menum{t}$ for his/her fateful guess (a color from $C$) which he then records as $\guess{t}$. Before $\menum{t}$ is asked at time $t$, he/she will be allowed to hear the record of all guesses $\guess{t'}$ such that $(t,t')\in H$ for certain directed graph $\stru{T}{H}$ such that $\prec\supseteq H$ and there is no infinite descending chain in $H$. Then $\menum{t}$'s memory is erased so that he doesn't remember what he heard nor what his own guess was for his future guesses.
 
%He then follows an enumeration (not necessarily one-to-one nor onto) $\menumf:\delta\rightarrow M$ of mathematicians, where $\delta$ is an ordinal number, and in the step $\alpha<\delta$ he asks the mathematician $\menum{\alpha}$ for his/her fateful guess (a color from $C$) which he then records as $\guess{\alpha}$. A mathematician $m$ will be allowed to hear the guess $\guess{\alpha}$ if and only if $(m,\alpha)\in H$ for certain binary relation $H\sbs M\times \delta$.

After the whole collection $\guessf:T\rightarrow C$ of guesses is obtained, the hatter compares it with the actual assignment $a$ using an evaluation function $e:A\times\fset{T}{C}\rightarrow 2$ and he sets the mathematicians free if $e(a,\guessf)=1$ or executes them if $e(a,\guessf)=0$.

\medskip

The hatter tells the mathematicians what $C$, $A$, $S$, $\stru{T}{H}$, $\menumf$ and $e$ he is going to use and leaves them for the night. They now have some time to find an optimal strategy that maximizes their chances of surviving.

\subsubsection{The formalization of the problem}

An instance (of the prisoners and hats problem) is an tuple $\mathcal{I}=(M,C,A,S,T,H,\menumf,e)$ where
\begin{itemize}
\item $M$ and $C$ are nonempty sets,
\item $\emptyset\neq A\sbs \fset{M}{C}$,
\item $S\sbs M^2$,
\item $\stru{T}{H}$ is a directed graph with no infinite descending paths,
\item $\menumf:T\rightarrow M$,
\item $e:A\times\fset{T}{C}\rightarrow 2$.
\end{itemize}

\subsubsection{Strategies}

There is a game $\mathcal{G}_{\mathcal{I}}$ of two players (the hatter and the group of mathematicians) with imperfect information that corresponds naturally (according to the story) to the instance $\mathcal{I}$.

We are not going to define $\mathcal{G}_{\mathcal{I}}$ here in detail. 
%For our purposes, the following facts are enough to state:
%\begin{enumerate}
%\item The players are $M$ (the group of mathematicians)  and $h$ (the hatter).
%\item $\mathcal{I}$ is known to both players.
%\item The first move of the game is some $a\in A$ chosen by the player $h$.
%\item All other moves are made by $M$ according to some well-ordering $\prec\supseteq H$ of $T$ (choice of which is not important for the game).
%\item In a move $t\in T$ the player $M$ chooses some $\guess{t}\in C$ based on the information $a\restr S^{-1}[\menum{t}]$ and $g\restr H^{-1}[t]$.
%\end{enumerate}
%
%In particular, a strategy of the player $M$ (for $\mathcal{G}_{\mathcal{I}}$ or just for $\mathcal{I}$) is any function $\sigma:(t,\alpha,\gamma)\mapsto c$ where $t\in T$, $\alpha: S^{-1}[\menum{t}]\rightarrow C$, $\gamma: H^{-1}[t] \rightarrow C$ and $c\in C$.
We just state, in a form of a definition, what a strategy for mathematicians in this game is: 

A strategy (for $\mathcal{I}$) is any function $\sigma:X\rightarrow C$ where $X$ is the set of all triples $(t,\alpha,\gamma)$ where $t\in T$, 
%$\alpha: S^{-1}[\menum{t}]\rightarrow C$ such that $\alpha\sbs\alpha'$ for some $\alpha'\in A$, 
$\alpha = \alpha' \restr S^{-1}[\menum{t}]$ for some $\alpha'\in A$,
$\gamma: H^{-1}[t] \rightarrow C$.
Any strategy $\sigma$ together with a choice of $a\in A$ uniquely determine the course of the game (i.e. the collection of all guesses $\guessf=\guessf_{a,\sigma}:T\rightarrow C$) and thus the result $e(a,\guessf)$:
\begin{lemma}
For any strategy $\sigma$ and $a\in A$ there is the unique $\guessf=\guessf_{a,\sigma}:T\rightarrow C$ such that for every $t\in T$
$$\guess{t}=\sigma(t,a\restr S^{-1}[\menum{t}],\guessf\restr H^{-1}[t]).$$
\end{lemma}
\begin{proof}
As $\stru{T}{H}$ does not contain infinite descending paths, there is a well-ordering $\prec\supseteq H$ on $T$. Let $(t_{\alpha})_{\alpha\in\delta}$ for some ordinal $\delta$ be an enumeration of $T$ such that $t_\alpha\prec t_\beta$ if and only if $\alpha<\beta$. We construct $\guessf:T\rightarrow C$ by recursion:
$$
\guess{t_\alpha}=\sigma(t_\alpha,a\restr S^{-1}[\menum{t_\alpha}],\guessf\restr H^{-1}[t_\alpha]),
$$
which is correct since $H^{-1}[t_\alpha]\sbs\prec^{-1}[t_\alpha]\sbs\set{t_\beta}{\beta<\alpha}$.

The uniqueness of $\guessf$ can be easily proven by induction.
\end{proof}

A strategy $\sigma$ for $\mathcal{I}$ is winning if $e(a,\guessf_{a,\sigma})=1$ for every $a\in A$.

\smallskip

Strategies for \uvz{disjoint} instances can be combined to give a strategy for their \uvz{union}: Let $I$ be a nonempty set and for $i\in I$ let $\sigma_i$ be a strategy for certain instance $\mathcal{I}_i=(M_i,C_i,A_i,S_i,T_i,H_i,\menumf_i,e_i)$. Assume that the sets $T_i$, for $i\in I$, are pairwise disjoint and let
\begin{itemize}
\item $M=\Suni_{i\in I}M_i$, 
\item $C=\Suni_{i\in I}C_i$,
\item $A\sbs\set{a\in\fset{M}{C}}{(\forall i\in I)(a\restr M_i\in A_i)}$,
\item $S\supseteq\Suni_{i\in I}S_i$,
\item $T=\Suni_{i\in I}T_i$,
\item $H\supseteq\Suni_{i\in I}H_i$,
\item $\menumf=\Suni_{i\in I}\menumf_i$,
\item $e:A\times\fset{T}{C}\rightarrow 2$.
\end{itemize}

Then there is a strategy $\sigma$ for the instance $\mathcal{I}=(M,C,A,S,T,H,\menumf,e)$ given by 
$$\sigma:(t,\alpha,\gamma)\mapsto \sigma_i(t,\alpha\restr S_i^{-1}[\menumf_i(t)],\gamma\restr H_i^{-1}[t])$$ 
where $i\in I$ is the unique index such that $t\in T_i$. We call such $\sigma$ the ($\mathcal{I}$-)combination of strategies $\sigma_i$, $i\in I$, and denote it by $\bigsqcup_{i\in I}^{\mathcal{I}}\sigma_i$ or just $\bigsqcup_{i\in I}\sigma_i$.

\section{A set-like case}
In this paper, we are going to deal only with cases when every mathematician is asked for her guess exactly once, i.e. when $\menumf:T\rightarrow M$ is a bijection. Then we can simplify our notation by assuming $T=M$, $\menumf=\id{T}$ and by omitting $T$ and $\menumf$ from the instances, writing them just as $\mathcal{I}=(M,C,A,S,H,e)$. Also we will always have $A=\fset{M}{C}$, which will allow for another simplification in notation later on.

The following evaluation functions will be used throughout the rest of this paper: 
For a cardinal number $\kappa$ we define evaluation functions $e_{\geq \kappa}$ (at least $\kappa$ correct guesses) and $e^{<\kappa}$ (less than $\kappa$ incorrect guesses) by 
\begin{align*}
e_{\geq \kappa}(a,\guessf)=1 & \IFF |\set{m\in M}{a(m)=\guess{m}}|\geq \kappa,\\
e^{<\kappa}(a,\guessf)=1 & \IFF |\set{m\in M}{a(m)\neq\guess{m}}|<\kappa.
\end{align*}

\medskip

The case, we are going to deal with in this section, requires no structure to be defined on the set $M$ (we call such cases \uvz{set-like}). This distinguishes this case from the cases defined in the following section, where we will be assuming a fixed well-ordering on $M$ (\uvz{well-ordered} cases).

\subsection{The \uvz{hear nothing, see all} case}
\label{sect:hnsa}
Let us denote by $\hnsa{M}{C}{A}{e}$ the instance $\mathcal{I}=(M,C,A,S,H,e)$ with $A=\fset{M}{C}$, $S=M^2\sdiff \id{M}$ and $H=\emptyset$ (that is every mathematician sees all hats except her own but hears no previous guesses). 

This is a basic case, many subcases of which have been thoroughly analyzed elsewhere. 
For the convenience of the reader, we present here a compact selfcontained overview that gives a complete (at least from our point of view) characterization of existence of winning strategies.

\subsubsection{Case: $M$ finite}
If both $M$ and $C$ are finite, then there is a strategy that guarantees at least $\lfloor|M|/|C|\rfloor$ correct guesses, but no strategy is able to guarantee more. This was probably first proven in \cite{Winkler01} for just two colors and in \cite{Feige04} in full generality (see also \cite[Theorem 2.3.1]{HTbook}):

\begin{prop}
\label{prop:dfovMfinCfin}
Let $M,C$ be finite, $n\in\N$. Then there is a winning strategy $\sigma$ for  the instance $\hnsa{M}{C}{A}{e_{\geq n}}$ if and only if $n\leq |M|/|C|$.
\end{prop}

\begin{proof}
\uvz{$\KAP$}: It is enough to find a winning strategy for $\hnsa{M'}{C'}{\fset{M'}{C'}}{e_{\geq 1}}$ where $|M'|=|C'|$. Indeed, $M$ can be divided into $n$ disjoint sets $M_0,\ldots,M_{n-1}$ of size $|C|$ and a possible leftover (as $n\leq |M|/|C|$). The combination $\sigma=\bigsqcup_{i<n}\sigma_i$ of winning strategies $\sigma_i$ for $\hnsa{M_i}{C}{\fset{M_i}{C}}{e_{\geq 1}}$ is then a winning strategy for $\hnsa{M}{C}{A}{e_{\geq n}}$. (There is at least one correct guess on each of $n$ disjoint sets $M_i$.)

Let us therefore assume that $|M|=|C|$ and $n=1$. Without loss of generality, we can have $M=C\in \N$. Define, for $m\in M$ and $\alpha:M\sdiff\lr{m}\rightarrow C$, the value of $\sigma$ as 
$$\sigma(m,\alpha,\emptyset) = m-\sum_{m'\neq m}\alpha(m'),$$ 
where the right side is computed modulo $M$. Then, for $a\in\fset{M}{C}$, let $m_a =\sum_{m'\in M} a(m')$ (modulo $M$). We get 
$$\guessf_{a,\sigma}(m_a)=\sigma(m_a,a\restr (M\sdiff\lr{m_a}),\emptyset)= m_a - \sum_{m'\neq m_a} a(m') = a(m_a),$$
thus guaranteeing $e(a,\guessf)=1$.

\uvz{$\PAK$}: Suppose that there is a winning strategy $\sigma$ for $n>|M|/|C|$. Then 
$$|\set{(a,m)}{a\in A, m\in M, a(m)=\guessf_{a,\sigma}(m)}|\geq n|A|>\frac{|M||A|}{|C|},$$ 
hence for some $m\in M$ we get 
$$|\set{a\in A}{a(m)=\guessf_{a,\sigma}(m)}|>|A|/|C|$$ 
and consequently 
$$|\set{a\in A}{a\supseteq a',a(m)=\guessf_{a,\sigma}(m)}|>1$$
for some $a':M\sdiff\lr{m}\rightarrow C$. But that means that for some $a_0\neq a_1$ from $A$ such that $a_i\supseteq a'$ for $i<2$ we have
$a_0(m)=\guessf_{a_0,\sigma}(m)=\sigma(m,a',\emptyset)=\guessf_{a_1,\sigma}(m)=a_1(m)$
and thus $a_0=a_1$ --- a contradiction.
\end{proof}

For infinite $C$, no strategy can guarantee even one correct guess:

\begin{prop}
Let $M$ be finite, $C$ be infinite. Then there is no winning strategy for $\hnsa{M}{C}{A}{e_{\geq 1}}$.
\end{prop} 
\begin{proof}
Suppose that a winning strategy $\sigma$ exists. Let $C'\sbs C$ such that $|C'|=|M|+1$. Clearly, existence of $\sigma$ implies existence of a winning strategy for $\hnsa{M}{C'}{\fset{M}{C'}}{e_{\geq 1}}$, which is impossible by the previous proposition.
\end{proof}

\subsubsection{Case: $M$ infinite}

If $|C|=1$, then obviously every strategy guarantees all guesses to be correct.
For infinite $M$ and $|C|>1$ there is a strategy that guarantees at most finitely many incorrect guesses, but no fixed finite number of incorrect guesses can be guaranteed. The implication \uvz{$\KAP$} in the following proposition is known as the Gabay-O'Connor theorem (see e.g., \cite[Theorem 3.2.2]{HTbook}):

\begin{prop}
\label{prop:hnsaMinf}
Let $M$ be infinite, $|C|>1$ and $\kappa$ a cardinal. There is a winning strategy for the instance $\hnsa{M}{C}{A}{e^{<\kappa}}$ if and only if $\kappa\geq\omega$.
\end{prop}

\begin{proof}
\uvz{$\KAP$}: Clearly, it is enough to find a winning strategy for $\hnsa{M}{C}{A}{e^{<\omega}}$. Denote $\sim$ the equivalence on $A$ such that $a\sim a'$ if and only if $a(m)\neq a'(m)$ for at most finitely many $m\in M$. Let $s:\factor{A}{\sim}\rightarrow A$ be a selector (i.e. $s([a]_{\sim})\sim a$ for all $a\in A$). Then a winning strategy can be defined as 
$$
\sigma(m,\alpha,\emptyset)=s([a]_{\sim})(m),
$$
for $m\in M$ and $\alpha:M\sdiff\lr{m}\rightarrow C$, where $a\in A$ is arbitrary such that $\alpha\sbs a$ (the definition is correct as any $a,a'\supseteq\alpha$ may have different value only at $m$ and thus $a\sim a'$). Indeed, $\sigma$ is winning, as for all but finitely many $m\in M$ we get 
$\guessf_{a,\sigma}(m)=\sigma(m,a\restr(M\sdiff\lr{m}),\emptyset)=s([a]_{\sim})(m)=a(m).$

\uvz{$\PAK$}: Suppose that there is a winning strategy $\sigma$ for $\hnsa{M}{C}{A}{e^{<n}}$ for some $n\in\N$. Choose $C'=\lr{c_0,c_1}\sbs C$ with $c_0\neq c_1$ and $M'\sbs M$ such that $|M'|=2n$.
Define a strategy $\sigma_0$ for $\hnsa{M'}{C'}{\fset{M'}{C'}}{e^{<n}}$ by $\sigma_0(m,\alpha,\emptyset)=\sigma(m,\alpha\suni ((M\sdiff M')\times \lr{c_0}),\emptyset)$ for $m\in M'$ and $\alpha:M'\sdiff\lr{m}\rightarrow C'$. Since $\sigma$ is winning, $\sigma_0$ is as well. That means that for each $a\in \fset{M'}{C'}$, we have $|\set{m\in M'}{a(m)=\guessf_{a,\sigma_0}(m)}|>|M'|-n = n$ and thus $\sigma_0$ is winning also for $\hnsa{M'}{C'}{\fset{M'}{C'}}{e_{\geq n+1}}$. But by Proposition \ref{prop:dfovMfinCfin} there is no winning strategy for $\hnsa{M'}{C'}{\fset{M'}{C'}}{e_{\geq k}}$ with $k>|M'|/|C'|=n$ --- a contradiction.
\end{proof}

\section{Well-ordered cases}

We now move to cases where the set $M$ is thought of as well-ordered and $S$ and $H$ are defined through this ordering.

Let $\M=\stru{M}{\prec}$ be a nonempty well-ordering. By $\succ$ we mean $\prec^{-1}$. We define two subcases of the problem of prisoners and hats -- the \uvz{hear nothing, see forward} case and the \uvz{hear backward, see-forward} case:
\begin{itemize}
\item $\hnsf{\M}{C}{A}{e}=(M,C,A,S,H,e)$ where $A=\fset{M}{C}$, $S=\succ$ and $H=\emptyset$,
\item $\hbsf{\M}{C}{A}{e}=(M,C,A,S,H,e)$ where $A=\fset{M}{C}$, $S=\succ$ and $H=\prec$.
\end{itemize}

In less formal words, the mathematicians are standing in a (well-ordered) line facing the (possibly infinite) tail of the line. In $\hnsf{\M}{C}{A}{e}$ every mathematician sees all hats in front of her but hears nothing, while in $\hbsf{\M}{C}{A}{e}$ she sees all hats in front of her and hears all guesses of her colleagues behind her.

Further on we assume (without loss of generality) that $M$ is a nonzero ordinal number and $\prec$ is the usual ordering $<$ of ordinals, i.e., $\M=\stru{\beta}{<}$ for an ordinal $\beta>0$.

\subsection{The \uvz{hear nothing, see forward} case}
\label{sect:hnsf}

\subsubsection{Case: $M$ finite}

For finite $M$ and $|C|>1$ no strategy guarantees even one correct guess:

\begin{prop}
\label{prop:hnsfMfin}
Let $|C|>1$ and $\M=\stru{n}{<}$ for $0<n\in\N$. Then there is no winning strategy for the instance $\hnsf{\M}{C}{A}{e_{\geq 1}}$.
\end{prop}

\begin{proof}
Assume, for contradiction, that $\sigma$ is a winning strategy for $\hnsf{\M}{C}{A}{e_{\geq 1}}$. We construct an assignment $a\in A$ such that $g_{a,\sigma}(m)\neq a(m)$ for all $m\in M$.
%, where $g_{a,\sigma}(m)$ is the guess of the player $m$ prescribed by $\sigma$.
This can be done easily by reverse induction on $n$: In the $(n-i)$th step we set $a(i)\neq\sigma(i,a\restr\lr{i+1,\ldots,n-1},\emptyset)=g_{a,\sigma}(i)$, which is possible thanks to $|C|>1$.
\end{proof}

\subsubsection{Case: $M$ infinite} 
If $\M=\stru{\alpha}{<}$ for $\alpha$ infinite, 
then the situation is similar to the \uvz{hear nothing, see all} case --- at most finitely many incorrect guesses can be guaranteed, but there is no finite upper bound on their number. The idea of the proof of the implication \uvz{$\KAP$} in the following proposition comes from \cite[Theorem 4.2.1]{HTbook}, however our setting is different:
%, but there is no finite upper bound on the number of errors

\begin{prop}
Let $|C|>1$, $\M=\stru{\beta}{<}$ for an ordinal $\beta\geq\omega$, and $\kappa$ be a cardinal. There is a winning strategy for the instance $\hnsf{\M}{C}{A}{e^{<\kappa}}$ if and only if $\kappa\geq\omega$.
\end{prop}

\begin{proof}
Analogous to the proof of Proposition \ref{prop:hnsaMinf}. 

\uvz{$\KAP$}: Clearly, it is enough to find a winning strategy $\sigma$ for $\hnsf{\M}{C}{A}{e^{<\omega}}$: For $m\in M$ we define the equivalence $\sim_m$ on $A=\fset{M}{C}$ by $a\sim_m a'$ if and only if $a(m')=a'(m')$ for all $m' > m$ (that is for all $m'$ that $m$ can see). We denote by $s$ the selector on $\poten{A}\sdiff\lr{\emptyset}$ that for $\emptyset \neq X\sbs A$ selects $s(X)$ as the $<^A$-least element of $X$, where $<^A$ is a fixed well-ordering of $A$. The strategy $\sigma$ can then be defined as 
$$
\sigma(m,\alpha,\emptyset)=s([a]_{\sim_m})(m),
$$
for $m\in M$ and $\alpha:\set{m'\in M}{m < m'}\rightarrow C$, where $a \in A$ is arbitrary such that $\alpha\sbs a$ (this is correct as any two such $a$'s are $\sim_m$-equivalent).

We prove that $\sigma$ is winning for $\hnsf{\M}{C}{A}{e^{<\omega}}$: For contradiction, let there be $a\in A$ and $M'\sbs M$ infinite such that for all $m\in M'$
\begin{equation}
\label{eq:proofhnsfinfkap}
a(m)\neq g_{a,\sigma}(m) = \sigma(m,a\restr\set{m'\in M}{m < m'},\emptyset)=s([a]_{\sim_m})(m).
\end{equation}

Let us choose an increasing sequence $m_0 < m_1 < \cdots$ of length $\omega$ in $M'$. We show that then $s([a]_{\sim_{m_0}})>^A s([a]_{\sim_{m_1}}) >^A \cdots$ is a strictly $<^A$-decreasing sequence in $A$, which is in contradiction with $<^A$ being a well-ordering:
From $m_i<m_{i+1}$ it follows $(\forall a'\in A)(a\sim_{m_i}a' \limp a\sim_{m_{i+1}}a')$, i.e., $[a]_{\sim_{m_i}}\sbs [a]_{\sim_{m_{i+1}}}$ and thus $s([a]_{\sim_{m_i}})\geq^A s([a]_{\sim_{m_{i+1}}})$ by the definition of $s$. But the equality is not possible since that would mean $g_{a,\sigma}(m_{i+1})=s([a]_{\sim_{m_{i+1}}})(m_{i+1}) = s([a]_{\sim_{m_i}})(m_{i+1})=a(m_{i+1})$, contradicting \eqref{eq:proofhnsfinfkap}.

\uvz{$\PAK$}: Assume for contradiction that there is a winning strategy $\sigma$ for $\hnsf{\M}{C}{A}{e^{<n}}$ for some $0<n\in\N$. Then we can take $\M'=\stru{n}{<}\sbs\M$ and
define a strategy $\sigma_0$ for $\hnsf{\M'}{C}{\fset{M'}{C}}{e^{<n}}$ by $\sigma_0(m,\alpha,\emptyset)=\sigma(m,\alpha\suni ((M\sdiff M')\times \lr{c_0}),\emptyset)$ for $m\in M'$ and $\alpha:\set{m'\in M'}{m < m'}\rightarrow C$, where $c_0\in C$ is fixed. From our assumption on $\sigma$ it follows that $\sigma_0$ is winning for $\hnsf{\M'}{C}{\fset{M'}{C}}{e^{<n}}$, and thus (because $|M'|=n$) also for $\hnsf{\M'}{C}{\fset{M'}{C}}{e_{\geq 1}}$. This contradicts Proposition \ref{prop:hnsfMfin}.
\end{proof}

\subsection{The \uvz{hear backward, see forward} case}
\label{sect:hbsf}

The approach we are going to present is an algebraic one. All the theorems that we use and that go beyond an elementary course of algebra can be found in \cite{Zie84}. 
The author owes the credit for the idea to Jan \v{S}aroch \cite{SarochNote}. 

Before stating the crucial Lemma \ref{lm:extendingsum}, let us recall some notions concerning Abelian groups. For an (additively written) Abelian group $G$ and an ordinal $\delta$, we denote by $G^\delta$ the product of $\delta$ copies of $G$, that is the group whose underlying set is $\fset{\delta}{G}=\set{f}{f:\delta\rightarrow G}$ and the group operation is defined pointwise. By $G^{(\delta)}$ we denote the direct sum of $\delta$ copies of $G$, that is the subgroup of $G^\delta$ whose underlying set is $\set{f}{f:\delta\rightarrow G\text{ and } f(\alpha)=0 \text{ for all but finitely many }\alpha\in\delta}$. There is the natural sum homomorphism $\Sigma:G^{(\delta)}\rightarrow G$ defined by $\Sigma(f)=f(\alpha_0)+\cdots +f(\alpha_{n-1})$ where $\alpha_{0},\ldots,\alpha_{n-1}$ is an enumeration of all $\alpha\in\delta$ for which $f(\alpha)\neq 0$ and the sum is computed in $G$.

\begin{lemma}
\label{lm:extendingsum}
For any ordinal $\delta>0$ and cardinal $\mu>0$ there is an Abelian group $G$ of size $\mu$ such that the sum homomorphism $\Sigma:G^{(\delta)}\rightarrow G$ can be extended to a homomorphism $\Sigma':G^\delta\rightarrow G$.
\end{lemma}

\begin{proof}
For $\mu\geq \omega$, let $G$ be a divisible Abelian group of size $\mu$, e.g. $G=\Q^{(\mu)}$. Then $G$ is injective in the category of Abelian groups and thus $\Sigma$ factors through the inclusion $\id{G^{(\delta)}}:G^{(\delta)}\rightarrow G^\delta$, yielding the required $\Sigma'$.

For $1<\mu<\omega$, we take $G=\Z/\mu\Z$ (the additive group of integers modulo $\mu$). Then $G$ is purely injective (because finite). The inclusion $\id{G^{(\delta)}}:G^{(\delta)}\rightarrow G^\delta$ is pure (even an elementary embedding, see \cite[Corollary 1.8]{Zie84}) and so $\Sigma$ factors through it, yielding $\Sigma'$ as before.
\end{proof}

Now we are ready to show that for the \uvz{hear backward, see forward} case there is always a strategy that guarantees at most one incorrect guess:

\begin{prop}
Let $|C|>1$, $\M=\stru{M}{\prec}$ be arbitrary, and $\kappa$ be a cardinal. There is a winning strategy for $\hbsf{\M}{C}{A}{e^{<\kappa}}$ if and only if $\kappa\geq 2$.
\end{prop}

\begin{proof}
\uvz{$\PAK$}: Suppose for contradiction that there is a winning strategy $\sigma$ for the instance $\hbsf{\M}{C}{A}{e^{<1}}$, that is that $\sigma$ guarantees all guesses correct. Then for the $\prec$-least element $m\in M$ we get $a(m)=g_{a,\sigma}(m)=\sigma(m,\alpha,\emptyset)=g_{a',\sigma}(m)=a'(m)$ for any $\alpha:M\sdiff\lr{m}\rightarrow C$ and $\alpha\subseteq a,a':M\rightarrow C$. This is of course a contradiction as we may take $a,a'$ such that $a(m)\neq a'(m)$.

\uvz{$\KAP$}: It is enough to find a winning strategy for $\hbsf{\M}{C}{A}{e^{<2}}$. The case $|M|=1$ is trivial. Further on we assume $|M|>1$. Without loss of generality suppose that $M=\delta\suni\lr{-1}$, for an ordinal $\delta>0$, $\prec$ is the usual ordering of $\delta\suni\lr{-1}$ and $C=\mu>0$ a cardinal. By Lemma \ref{lm:extendingsum} there are an Abelian group $G$ of size $\mu$ and a homomorphism $\Sigma':G^{\delta}\rightarrow G$ extending the sum homomorphism $\Sigma: G^{(\delta)}\rightarrow G$. Again, without loss of generality, we may assume that the underlying set of $G$ is $\mu=C$. 

We define $\sigma$ as follows: 
$$\sigma(-1,\alpha,\emptyset)=\Sigma'(\alpha)\text{ for }\alpha\in G^\delta=\fset{M}{C}\text{, and }\sigma(\beta,\alpha,\gamma)=\gamma(-1)-\Sigma'(\alpha\suni\lr{(\beta,0)}\suni\gamma\restr\beta),$$
for $\beta\in\delta$, $\alpha:\set{\beta'}{\beta<\beta'\in\delta}\rightarrow C$, and $\gamma:\beta\suni\lr{-1}\rightarrow C$, where the subtraction is computed in $G$.

We prove that 
\begin{equation}
\label{eq:proofhbsf}
g_{a,\sigma}(\beta)=a(\beta)
\end{equation}
 for all $\beta\in\delta$ and $a\in A$ by induction on $\beta$: First, $g_{a,\sigma}(-1)=\sigma(-1,a\restr\delta,\emptyset)=\Sigma'(a\restr\delta)$ (which may differ from $a(-1)$). For the inductive step, let us assume \eqref{eq:proofhbsf} for all $\beta'<\beta$. Then 
\begin{align*}
g_{a,\sigma}(\beta) &=\sigma(\beta,a\restr\set{\beta'}{\beta<\beta'\in\delta},g_{a,\sigma}\restr(\beta\suni\lr{-1}))= \\
&=g_{a,\sigma}(-1)-\Sigma'(a_\beta)= \\ 
&=\Sigma'(a\restr\delta)-\Sigma'(a_\beta)= \\
&=\Sigma'(a\restr\delta-a_\beta)= \\
&=\Sigma(a\restr\delta-a_\beta)= \\
&=a(\beta),
\end{align*} 
where $a_\beta=a\restr(\delta\sdiff\lr{\beta})\suni\lr{(\beta,0)}$, the first equality is just definition of $g_{a,\sigma}$, the second one holds due to the inductive assumption, the third one holds by the \uvz{$-1$st step}, and the last three follow from the properties of $\Sigma$ and $\Sigma'$ and the definition of $a_\beta$.
\end{proof}

%\section{Binary tree case}

\bibliography{bibliography}{}
\bibliographystyle{amsalpha}

\end{document}

%% file: packages.tex
\usepackage{a4wide} % nastavuje standardní evropský formát stránek A4
\usepackage[stable]{footmisc} % umoznuje pridavani footnotes i do nazvu kapitol a sekci
\usepackage{fancyhdr} % zahlavi stranek, zapina se \pagestyle{fancy}
\usepackage{layout} % umoznuje prikaz \layout pro zobrazeni diagramu layoutu (nakresleni marginu, paddingu, atd...)

\usepackage{amsthm,amsmath,amstext,amsbsy,amsfonts,amssymb} % AMS
\usepackage{stmaryrd} % extra matematicke symboly, doplnek k AMS
\usepackage{turnstile} % snadne vlastni kresleni symbolu typu dokazovatka, pravdivitka, forsitka, a mnoha dalsich...
\usepackage{mathtools} % rozsireni amsmath
\usepackage{graphicx} % nezbytné pro standardní vkládání obrázků do dokumentu
\usepackage[all,cmtip]{xy} % kresleni diagramu xymatrix, Xy-pic
\usepackage{fancybox} % umožňuje pokročilé rámečkování :-)
\usepackage{hhline} % kresleni vertikalnich i svislych tabulkovych car/dvojcar
\usepackage{colortbl} % barevne tabulky
\usepackage{color} % barevny text
\usepackage{pstricks} % PStricks - kresleni obrazku
\usepackage{qtree} % jednoduche kresleni stromu
\usepackage{tikz} % kresleni komplikovanych diagramu
\usepackage{enumerate} % rozsirene moznosti enumeraci
\usepackage[ps2pdf]{hyperref} % hypertextove odkazy
\usepackage{index} % nutno pouľít v případě tvorby rejstříku balíčkem makeindex
\usepackage{longtable} % vicestrankove tabulky
\usepackage{cite} % bibliografie

\usepackage{ifthen} % if, then, else
\usepackage{xstring} % prace s retezci znaku
\usepackage{calc} % zakladni aritmetika intuitivne
\usepackage{forloop} % implementace for cyklu

\usepackage{xspace} %\xspace se dava na konec makra, ktere bude pouzivane prevazne v textu, prida pak za jeho uziti mezeru